\documentclass[11pt]{article}
\usepackage{amsfonts,amssymb,amsmath,latexsym,color,epsfig,url}
\newtheorem{theorem}{Theorem}
\newtheorem{lemma}[theorem]{Lemma}

\newtheorem{conjecture}{Conjecture}
\newtheorem{corollary}[theorem]{Corollary}
\newtheorem{question}{Question}

\newenvironment{proof}
      {\medskip\noindent{\bf Proof:}\hspace{1mm}}
      {\hfill$\Box$\medskip}

\def\qed{\ifvmode\mbox{ }\else\unskip\fi\hskip 1em plus 10fill$\Box$}
\input{epsf}
\makeatletter
\def\Ddots{\mathinner{\mkern1mu\raise\p@
\vbox{\kern7\p@\hbox{.}}\mkern2mu
\raise4\p@\hbox{.}\mkern2mu\raise7\p@\hbox{.}\mkern1mu}}
\makeatother

\title{\vspace{-0.7cm}Stanley-Wilf limits are typically exponential}
\author{
Jacob Fox\thanks{
    Department of Mathematics,
    Massachusetts Institute of Technology,
    Cambridge, MA 02139-4307.
    Email: {\tt fox@math.mit.edu}.
    Research supported by a Packard Fellowship, by a Simons Fellowship, by NSF grant DMS-1069197, by an Alfred P. Sloan Fellowship, and by an MIT NEC Corporation Award.}}

\date{}
\begin{document}
\maketitle

\begin{abstract} 
For a permutation $\pi$, let $S_{n}(\pi)$ be the number of permutations on $n$ letters avoiding $\pi$. Marcus and Tardos proved the celebrated Stanley-Wilf conjecture that $L(\pi)= \lim_{n \to \infty} S_n(\pi)^{1/n}$ exists and is finite. Backed by numerical evidence, it has been conjectured by many researchers over the years that $L(\pi)=\Theta(k^2)$ for every permutation $\pi$ on $k$ letters. We disprove this conjecture, showing that $L(\pi)=2^{k^{\Theta(1)}}$ for almost all permutations $\pi$ on $k$ letters.
\end{abstract}

\section{Introduction}

Pattern avoidance is a central topic in combinatorics. Permutation avoidance has been a particularly popular area of study. This can be seen from the books \cite{B12,Ki1} and surveys \cite{S06,Ste13}, the annual conference {\it Permutation Patterns} since 2003, and the many applications collected in Tenner's database \cite{Te}. A permutation of $[n]:=\{1,\ldots,n\}$ is called an {\it $n$-permutation}. An $n$-permutation $\sigma$ {\it contains} a $k$-permutation $\pi$ if there exists integers $1 \leq x_1 < x_2 < \ldots < x_k \leq n$ such that for $1 \leq i,j \leq k$ we have $\sigma(x_i) < \sigma(x_j)$ if and only if $\pi(i) < \pi(j)$. Otherwise, $\sigma$ {\it avoids} $\pi$. 

For a permutation $\pi$, let $S_{n}(\pi)$ be the number of $n$-permutations avoiding $\pi$. Classical results of McMahon \cite{M15} and Knuth \cite{K68} imply that for every $3$-permutation $\pi$ and every positive integer $n$, we have  $S_n(\pi)$ is the $n$th Catalan number $\frac{1}{n+1}{2n \choose n}$. A consequence of the RSK algorithm is that, for $\pi=12\cdots k$ the identity $k$-permutation, $$\lim_{n \to \infty} S_n(\pi)^{1/n} =(k-1)^2,$$ and Regev \cite{R81} proved a stronger asymptotic formula (see also \cite{N11}). 

Stanley and Wilf independently (see \cite{Notices} for a complete history) asked in 1980 about the behavior of $S_n(\pi)$ for a general $k$-permutation $\pi$ and large $n$. Wilf was originally unaware of Regev's work and asked if $S_n(\pi) \leq (k+1)^n$, while Stanley asked if $\lim_{n \to \infty} S_n(\pi)^{1/n}=(k-1)^2$.  
Both of these original questions have negative answers. They quickly modified these questions to the following conjecture: For every $k$-permutation $\pi$ there is a finite number $L(\pi)$ such that $\lim_{n \to \infty}  S_n(\pi)^{1/n}=L(\pi)$. 

A seemingly weaker conjecture considered by B\'ona and others asks if, for every permutation $\pi$, there exists $C=C(\pi)$ such that $S_{n}(\pi) \leq C^n$ for all $n$.  As observed by Arratia \cite{A99}, these two conjectures are equivalent. This equivalence follows from the simple observation that $S_n(\pi)$ is super-multiplicative. Indeed, by symmetry, we may assume the first letter in $\pi$ is larger than the last letter in $\pi$. The super-multiplicativity then follows from the fact that the concatenation of two permutations which avoid $\pi$ where every letter in the first permutation is smaller than every letter in the second permutation also avoids $\pi$. These equivalent conjectures became known as the {\it Stanley-Wilf conjecture}, a name introduced by B\'ona \cite{Notices}.

Alon and Friedgut \cite{AF} conjectured that the longest word avoiding a fixed $k$-permutation $\pi$ and satisfying that equal letters are distance at least $k$ in the word has length linear in the alphabet size. They showed their conjecture implies the Stanley-Wilf conjecture. They use this relationship to get a slightly super-exponential bound on $S_n(\pi)$, of the form $C(\pi)^{n\gamma(n)}$, where $\gamma(n)$ is an extremely slow growing function related to the inverse Ackerman hierarchy. Klazar \cite{K00} proved that the Stanley-Wilf conjecture is implied by the F\"uredi-Hajnal conjecture \cite{FH}, an extremal problem for matrices described below. He also showed that the Alon-Friedgut conjecture is equivalent to the F\"uredi-Hajnal conjecture. 

Marcus and Tardos \cite{MT} proved by an elegant argument the F\"uredi-Hajnal conjecture and hence the Stanley-Wilf conjecture and the Alon-Friedgut conjecture. This important work has led to a great deal of further developments. The number $L(\pi)=\lim_{n \to \infty}S_n(\pi)^{1/n}$ is known as the {\it Stanley-Wilf limit} of the permutation $\pi$. For a $k$-permutation $\pi$, the Marcus-Tardos proof of the Stanley-Wilf conjecture shows that $L(\pi) \leq 15^{2k^4{k^2 \choose k}}$.  

 In 1999, Arratia \cite{A99} conjectured that the quadratic bound $L(\pi) \leq (k-1)^2$ holds for every $k$-permutation $\pi$. In the tradition of Erd\H{o}s, Arratia further offered $\$ 100$ for settling this conjecture.  Seven years later, Albert et al.~\cite{AERWZ} (see also \cite{B07}) disproved the conjectured bound by a bit. They showed $L(4231)>9.47$, whereas the conjectured upper bound was $9$.

Since the work of Marcus and Tardos, the problem of closing the large gap between the quadratic and double-exponential bounds has attracted a great deal of further attention. Here we mention a few of these developments.

It has been conjectured by many researchers that $L(\pi)=\Theta(k^2)$ for every $k$-permutation $\pi$. A permutation is {\it layered} if it is a concatenation of decreasing sequences, the letters of each sequence being smaller than the letters in the following sequences. Backed by numerical evidence, first computed by West \cite{W90} and later replicated by many others, Bona \cite{B05} (see also \cite{B12,B13+,C09,CJS,GM14,Ste13}) conjectured that among the patterns of a given length, the largest Stanley-Wilf limit is attained by a layered permutation. The recent survey \cite{Ste13} states that this conjecture is widely believed to be true. Claesson, Jel\'inek, and Steingr\'imsson \cite{CJS} proved (see also \cite{B13+}) that if $\pi$ is a layered $k$-permutation, then $L(\pi) \leq 4k^2$. Valtr (see \cite{KK02}) showed that there is an absolute positive constant $c$ such that $L(\pi) \geq ck^2$ holds for every permutation $\pi$ on $k > 2$ letters. Thus the conjecture that $L(\pi)=\Theta(k^2)$ for every $k$-permutation $\pi$ would follow from B\'ona's conjecture. 

We disprove these conjectures on Stanley-Wilf limits, showing that for each $k$, there is a $k$-permutation $\pi$ such that $L(\pi)$ has exponential-type growth in $k$.

\begin{theorem}\label{main1}
For each $k$, there is a $k$-permutation with $L(\pi)=2^{\Omega(k^{1/4})}$.
\end{theorem}

With a slightly weaker bound, we can simultaneously avoid almost all $k$-permutations. For a family $U$ of permutations, let $S_n(U)$ be the number of $n$-permutations which avoid all permutations in $U$. Stanley \cite{St08} has asked if for each finite set $U$ of permutations, $\lim_{n \to \infty} S_n(U)^{1/n}$ exists. If this limit exists, we denote it by $L(U)$. %Arratia's observation described above that $L(\pi)$ exists extends to show that $L(U)$ exists if no permutation in $U$ is the concatenation of two permutations, where every letter in the first permutation is smaller than every letter in the second permutation.

\begin{theorem}\label{mainfam}
For each $k$, there is a family $U$ consisting of almost all $k$-permutations such that $L(U)$ exists and satisfies $L(U) =2^{\Omega\left((k/\log k)^{1/4}\right)}$.
\end{theorem}  

A matrix is {\it binary} if its entries are $0$ or $1$. All matrices we consider in this paper are binary. Matrix $A$ {\it contains} a $k \times \ell$ matrix $P=(p_{ij})$ if there exists a $k \times \ell$ submatrix $D=(d_{ij})$ of $A$ with $d_{ij}=1$ whenever $p_{ij}=1$. Otherwise we say that $A$ {\it avoids} $P$. 

The {\it mass} of a  matrix is the number of its one-entries. Equivalently, the mass of a matrix is the sum of all its entries. Let $\textrm{ex}(n,P)$ be the maximum possible mass of an $n \times n$  matrix that avoids $P$. For a permutation $\pi$ with matrix $P$, we say {\it $A$ avoids $\pi$} if it avoids $P$ and we let $\textrm{ex}(n,\pi)=\textrm{ex}(n,P)$. F\"uredi and Hajnal conjectured that, for each permutation $\pi$, we have $\textrm{ex}(n,\pi) =O(n)$. 

The function $\textrm{ex}(n,\pi)$ is super-additive. Indeed, by symmetry, we may assume the first letter in $\pi$ is larger than the last letter in $\pi$, and then the direct sum of two matrices which avoid $\pi$ also avoids $\pi$. More generally, Pach and Tardos (Lemma 1(ii) in \cite{PT06}) showed that $\textrm{ex}(n,P)$ is super-additive. Marcus and Tardos \cite{MT} proved that  \begin{equation}\label{MTineq}\textrm{ex}(n,\pi) \leq 2k^4{k^2 \choose k}n\end{equation} holds for every $k$-permutation $\pi$. It follows from this linear bound and the fact that $\textrm{ex}(n,\pi)$ is super-additive that, as $n$ tends to infinity,  $\frac{\textrm{ex}(n,\pi)}{n}$ tends to a finite limit $c(\pi)$. The number $c(\pi)$ is known as the {\it F\"uredi-Hajnal} limit of $\pi$.  The upper bound (\ref{MTineq}) implies $c(\pi)=2^{O(k\log k)}$.

Klazar's proof \cite{K00} that the F\"uredi-Hajnal conjecture implies the Stanley-Wilf conjecture shows that $L(\pi)=2^{O(c(\pi))}$. Cibulka \cite{C09} recently examined the relationship between the Stanley-Wilf limit $L(\pi)$ and the F\"uredi-Hajnal limit $c(\pi)$, showing that they are polynomially related. In one direction he proved  $c(\pi)=O(L(\pi)^{4.5})$. In the other direction, he improved Klazar's upper bound on $L(\pi)$ to $L(\pi)=O(c(\pi)^2)$. A simple proof of this result is given in Section \ref{latersect}. This result implies the improved bound $L(\pi)=2^{O(k \log k)}$ on the Stanley-Wilf limit. Thus, Theorem \ref{mainfam} shows that the Stanley-Wilf limits are typically exponential. 

To prove Theorems \ref{main1} and \ref{mainfam}, in Section \ref{sectlb} we construct a very dense matrix of exponential size  which avoids almost all $k$-permutations. By super-additivity of $\textrm{ex}(n,\pi)$, this implies a lower bound on $c(\pi)$ and hence we get a lower bound on $L(\pi)$ as well. 

We also improve the upper bound on $c(\pi)$ and $L(\pi)$. 

\begin{theorem}\label{newthm}
For every $k$-permutation $\pi$, we have $c(\pi)=2^{O(k)}$ and $L(\pi)=2^{O(k)}$. 
\end{theorem}

As discussed in Section \ref{sectconc}, this improvement on the Marcus-Tardos bound also implies an improved running time of $2^{O(k^2)}n$ on the Guillemot-Marx algorithm \cite{GM14} for determining whether an $n$-permutation $\sigma$ contains a $k$-permutation $\pi$. 

{\bf Organization} In the next section, we introduce the notion of interval minors of a matrix, and relate it to containment of a permutation matrix. We then prove Theorem \ref{main1} in Section \ref{sectlb}. In Section \ref{latersect}, we give a new simple proof of a result of Cibulka \cite{C09} giving an upper bound on the Stanley-Wilf limit which is quadratic in the F\"uredi-Hajnal limit. In Section \ref{sectnewthm}, we prove Theorem \ref{newthm}, which gives an improved upper bound on Stanley-Wilf limits. In Section \ref{sectconc}, we present some concluding remarks and open problems. 

All logarithms are base $2$ unless otherwise stated. For the sake of clarity of presentation, we systematically omit all floor and ceiling signs whenever they are not crucial. We also do not make any serious attempt to optimize constants in our statements and proofs. 

\section{Interval minors} 

Many combinatorial problems concern containment of substructures in larger structures. For example, in graph theory, some common containments studied include subgraph, induced subgraph, minor, topological minor or subdivision, and immersion.  Here we will study an analogue of graph minor for matrices, where instead of contracting adjacent vertices, we consider contracting consecutive rows or columns of the matrix. 

The {\it interval contraction} of two consecutive rows of a  matrix replaces the two rows by a single row, placing a one in an entry of the new row if at least one of the two entries in the original two rows is a one, and otherwise placing a zero in that entry of the new row. Interval contraction of two consecutive columns is defined similarly. A  matrix $P=(p_{ij})$ is an {\it interval minor} of another  matrix $A=(a_{ij})$ if $P$ is contained in a matrix obtained from $A$ by interval contraction. We say $A$ {\it avoids $P$ as an interval minor} if $P$ is not an interval minor of $A$. 

Equivalently, a $k \times \ell$ matrix $P$ is an interval minor of a matrix $A$ if 
\begin{itemize} 
\item there are $k$ disjoint intervals of rows $I_1,\ldots,I_k$ with $I_i$ coming before $I_j$ if $i<j$, 
\item $\ell$ disjoint intervals of columns $L_1,\ldots,L_{\ell}$ with $L_i$ coming before $L_j$ if $i<j$, 
\item and for all $(a,b) \in [k] \times [\ell]$, if $p_{ab}=1$, then the submatrix $I_a \times L_b$ of $A$ contains a one entry. 
\end{itemize} 
An {\it interval of rows} (columns) is a set of consecutive rows (columns). 
By enlarging the intervals of rows if possible, in the above definition we can restrict to sets of intervals of rows that form a partition of the set of rows, and similarly we can restrict to sets of intervals of columns that form a partition of the set of columns. 

This notion has an analogue in graph minors. We may view a  matrix as a bipartite graph with the set of rows and the set of columns as the two parts, with an adjancency between a row and a column if their common entry is a one. The standard notion of contraction in graphs replaces two adjacent vertices by a single vertex whose neighborhood is the union of the neighborhoods of the two vertices it replaced.  For comparison, interval contraction replaces two consecutive vertices by a single vertex whose neighborhood is the union of the neighborhoods of the two vertices. Thus, interval contraction replaces ``adjacent'' by ``consecutive''.

As is standard, we use $J_k$ to denote the $k \times k$ matrix which is all ones.  Of course, $J_k$ contains every $k$-permutation. The following lemma is a partial converse of this fact.

\begin{lemma}\label{lemexists}
There is an $\ell^2$-permutation whose matrix contains $J_{\ell}$ as an interval minor. 
\end{lemma}
\begin{proof}
Consider the $\ell^2$-permutation $\pi$ defined by $\pi(a\ell+b+1)=b\ell+a+1$ for $0 \leq a,b \leq \ell-1$. Partitioning the set of rows and the set of columns of the permutation matrix $A$ of $\pi$ into intervals of length $\ell$, each of the $\ell \times \ell$ blocks has a one in it. Hence, contracting these intervals, we get that $A$ contains $J_\ell$ as an interval minor.
\end{proof}

Note that the bound $\ell^2$ in the above lemma cannot be decreased. Indeed, if a matrix $P$ is an interval minor of another matrix $A$, then the mass of $A$ is at least the mass of $P$. Since the mass of $J_{\ell}$ is $\ell^2$, any matrix which contains $J_{\ell}$ as an interval minor must have mass at least $\ell^2$. Hence, if a $k$-permutation contains $J_{\ell}$, then $k \geq \ell^2$.

The next lemma shows that random permutations of size a logarithmic factor larger than in the previous lemma almost surely contain the complete matrix $J_\ell$ as an interval minor. 

\begin{lemma}\label{lemrandom} For $k \geq 3\ell^2\ln \ell$, almost all $k$-permutations contain $J_{\ell}$ as an interval minor. 
\end{lemma}
\begin{proof}
In the matrix $A$ of a random $k$-permutation $\pi$, the probability that a given $(k/\ell) \times (k/\ell)$ submatrix has all zeros is at most $(1-1/\ell)^{k/\ell}<e^{-k/\ell^2}$.  Thus, if $k \geq 3\ell^2 \ln \ell$, then this probability is less than $\ell^{-3}$. If the rows of $A$ are partitioned into $\ell$ equal intervals and the columns of $A$ are partitioned into $\ell$ equal intervals, then we have $\ell^2$ blocks, and we get the probability that $A$ avoids $J_{\ell}$ as an interval minor is at most $\ell^2\ell^{-3}=1/\ell$, completing the proof.
\end{proof}

Noga Alon pointed out that the bound in the above lemma is tight up to the constant factor. Indeed, a first moment argument shows that the above lemma is not true for $k<c\ell^2 \ln \ell$, where $c$ is a small positive constant. 

\section{Lower bound construction} \label{sectlb}

We prove the following theorem, which we subsequently show implies Theorems \ref{main1} and \ref{mainfam}. 

The {\it interval} $[a,b]:=\{a,a+1,\ldots,b\}$ consists of all integers between $a$ and $b$. For brevity, we often write $[b]:=[1,b]$. A {\it dyadic interval} is an interval of the form $[(s-1)2^t+1,s2^t]$, where $s$ and $t$ are nonnegative integers. A {\it rectangle} is a product $[a_1,b_1] \times [a_2,b_2]= \{(x,y): x \in [a_1,b_1]~\textrm{and}~y \in [a_2,b_2]\}$ of two intervals. A {\it dyadic rectangle} is a product of two dyadic intervals. 

\begin{theorem}\label{mainlater}
Let $r,\ell$ be positive integers and $0<q<1/2$ with $3 \leq r \leq q\ell/4$. Let $N=2^r$. There is an $N \times N$ matrix $M$ with mass at least $(1-q)^{(r+1)^2}N^2-1$
 which avoids $J_{\ell}$ as an interval minor. 
\end{theorem}
\begin{proof} Let $\mathcal{I}$ denote the collection of all dyadic intervals $I \subset [N]$, and $\mathcal{S}$ be the collection of all dyadic rectangles $R \subset [N] \times [N]$. Note that each $i \in [N]$ is in exactly  $r+1$ intervals in $\mathcal{I}$, so each entry of $M$ is in exactly $(r+1)^2$ rectangles in $\mathcal{S}$. Let $\mathcal{R}$ be a random subcollection of $\mathcal{S}$, where each dyadic rectangle appears in $\mathcal{R}$ with probability $1-q$, independently of the other dyadic rectangles. Let $M$ be the $N \times N$ matrix where an entry of $M$ is one if each of the $(r+1)^2$ rectangles in $\mathcal{S}$ containing it are also in $\mathcal{R}$, and zero otherwise. It follows that each entry of $M$ is one with probability $(1-q)^{(r+1)^2}$. By linearity of expectation, the expected mass of $M$ is $(1-q)^{(r+1)^2}N^2$. 

Let $N'=|\mathcal{I}|$, so $N'=\sum_{i=0}^r 2^i = 2N-1$. We also consider an auxiliary $N' \times N'$ matrix $B$, which has a row for each $I \in \mathcal{I}$ and a column for each $J \in \mathcal{I}$, and the $(I,J)$ entry of $B$ is one if $I \times J \in \mathcal{R}$ and zero otherwise. Hence, each entry of $B$ is one with probability $1-q$, independently of the other entries. 

Let $J_{\ell}(B)$ denote the number of copies of $J_{\ell}$ in $B$. Consider the random variable $$X: = \textrm{mass of }M \, - \, N^2 J_{\ell}(B).$$ By linearity of expectation, we have $$\mathbb{E}[X]=(1-q)^{(r+1)^2}N^2-N^2{N' \choose \ell}^2(1-q)^{\ell^2} >(1-q)^{(r+1)^2}N^2- N^{2\ell+2}e^{-q\ell^2} > (1-q)^{(r+1)^2}N^2-N^{-\ell},$$ 
where we used $\ell \geq 4$, $e^{-q} > 1-q$, $N=2^r$, and $r \leq q\ell/4$.

Fix a choice of $\mathcal{R}$ with $X \geq \mathbb{E}[X]$. Note that $X>0$ as $(1-q)^{(r+1)^2}>2^{-4qr^2}=N^{-4qr} \geq N^{-\ell}$.
Since $X>0$, it follows that the number of copies of $J_{\ell}$ in $B$ is $0$, i.e., $B$ avoids $J_{\ell}$. Also, the mass of $M$ is $X$. 

We will use the fact that $B$ avoids $J_{\ell}$ to show that $M$ avoids $J_{\ell}$ as an interval minor. 
Suppose for the sake of contradiction that $M$ contains $J_{\ell}$ as an interval minor, so there are disjoint intervals of rows $I_1, \ldots, I_{\ell}$  of $M$ and disjoint intervals of columns $L_1,\ldots,L_{\ell}$ of $M$, such that for each $(a,b) \in [\ell]^2$, the submatrix of $M$ with row set $I_a$ and column set $L_b$  contains at least one one-entry. 

We associate to each interval of rows $I_a$ the smallest dyadic interval $v_a \in \mathcal{I}$ that is a superset of $I_a$, and to each interval of columns $L_b$ the smallest dyadic interval $w_b \in \mathcal{I}$ that is a superset of $L_b$. 

The dyadic intervals $v_1,\ldots,v_{\ell}$ are distinct, and similarly, the dyadic intervals $w_1,\ldots,w_{\ell}$ are distinct. Indeed, this follows from the fact that if an interval $I$ is partitioned into two subintervals $I'$ and $I''$, and $I_a$ and $I_b$ are disjoint subintervals of $I$, then at least one of $I_a$ or $I_b$ is a subset of $I'$ or $I''$. 

As there is a one in the submatrix with row set $I_a$ and column set $L_b$, $I_a \subset v_a$, and $L_b \subset w_b$, then the $(v_a,w_b)$ entry in $B$ must be a one. Therefore, $B$ contains $J_{\ell}$ as a submatrix with rows $v_1,\ldots,v_{\ell}$ and columns $w_1,\ldots,w_{\ell}$. This contradicts that $B$ avoids $J_{\ell}$, and completes the proof. 
\end{proof}

Noga Alon had the nice idea of using the random variable $X$ in the proof. An earlier write-up showed that the probability that the mass of $M$ is large is greater than the probability that $B$ contains $J_{\ell}$. 

We think that the use of random dyadic rectangles, as in the above proof, might be useful for other ordered extremal problems as well. A different model of random dyadic rectangles, where the rectangles are of equal area and form a tiling, was first considered by Janson, Randall, and Spencer \cite{JRS}, and in the recent paper \cite{Ang} (see also \cite{CLSW}) . 

Applying Lemma \ref{lemexists} with $\ell=k^{1/2}$, there is a $k$-permutation $\pi$ which avoids $J_{\ell}$. From Theorem \ref{mainlater} with $q=\ell^{-1/2}$ and $r=\ell^{1/2}/8$, we get the following corollary. Indeed, note that $N=2^r=2^{\Omega(k^{1/4})}$ and the mass of the matrix $M$ we get in Theorem \ref{mainlater} is at least $(1-q)^{(r+1)^2}N^2-1 > 2^{-3qr^2}N^2-1= N^{2-3qr}-1>N^{3/2}$. 

\begin{corollary}\label{cor9}
For each $k>2$ there is a permutation $\pi$ on $k$ elements and an $N \times N$ matrix $M$ with $N=2^{\Omega(k^{1/4})}$ such that the mass of $M$ is at least $N^{3/2}$ and $M$ avoids $\pi$. 
\end{corollary}

As $\textrm{ex}(n,\pi)$ is super-additive, we  get $c(\pi) \geq \frac{\textrm{ex}(N,\pi)}{N} \geq N^{1/2}$ with $N=2^{\Omega(k^{1/4})}$. 

\begin{corollary}\label{cor10}
For each $k$, there is a $k$-permutation $\pi$ with $c(\pi)=2^{\Omega(k^{1/4})}$. 
\end{corollary}

By Cibulka's result that $c(\pi)=O(L(\pi)^{4.5})$, Theorem \ref{main1} follows. 

We next show a simpler deduction of the weaker estimate $L(\pi) = 2^{\Omega(k^{1/6})}$. From Lemma \ref{lemexists} with $\ell=k^{1/2}$ and Theorem \ref{mainlater} with $q=\ell^{-1/3}$ and $r=\ell^{1/3}/1000$, we obtain that there is a $k$-permutation $\pi$ and an $N \times N$ matrix $M$ with $N=2^{\Omega(k^{1/6})}$ and density at least $.99$ which avoids $\pi$. By repeatedly deleting a row or column with density less than $.9$ together with an arbitrary row or column so as to keep it a square matrix, we can find a $N' \times N'$ submatrix $M'$ of $M$ with $N' \geq .8N$ so that every row and column of $M'$ has density at least $.9$.  The number of zero entries deleted at step $i$ is more than $.1(N-i)$, and $M$ has at most $.01N^2$ zeros. If there are $s$ total steps, then the number of zeros deleted is least  $\sum_{i=0}^{s-1} .1(N-i)  = .1Ns-{s \choose 2}$, which is at most $.01N^2$, implying that $s$ is at most $.2N$. The resulting submatrix $M'$ is $N' \times N'$ with $N' \geq N-.2N=.8N$ and every row and column has density at least $.9$.  

The problem of counting permutation matrices contained in a matrix is equivalent to counting perfect matchings in the corresponding bipartite graph, with rows and columns as vertices, and a row is adjacent to a column if and only if their common entry is a one. One can arbitrarily start the permutation by picking the ones in the first $.3N'$ rows, giving at least $(.5N')^{.3N'}$ possible choices. By Hall's matching theorem, the partial permutation can be completed to a permutation, giving at least as many possible permutations. This gives $S_{N'}(\pi) \geq (.5N')^{.3N'}$.  The estimate $L(\pi) = 2^{\Omega(k^{1/6})}$ follows from the fact that $S_n(\pi)$ is super-multiplicative. 

For a family $U$ of permutations, let $\textrm{ex}(n,U)$ be the maximum mass of an $n \times n$ matrix which avoids every permutation in $U$. Let $c(U)=\lim_{n \to \infty} \frac{\textrm{ex}(n,U)}{n}$ if this limit exists. Notice that the fraction of $k$-permutations which are the concatenation of two permutations, where every letter in the first permutation is smaller than every letter in the second permutation, tends to $0$ as $k$ tends to infinity. Let $U$ be the family of all $k$-permutations $\pi$ which is not the concatenation of two permutations, where every letter in the first permutation is smaller than every letter in the second permutation, and $\pi$ contains $J_{\ell}$ as an interval minor, where $\ell=\left(\frac{k}{3\ln k}\right)^{1/2}$. By Lemma \ref{lemrandom} and the discussion above, almost all $k$-permutations are in $U$. Also, by the same arguments given in the introduction on super-multiplicitivity of $S_n(\pi)$ and the super-additivity of $\textrm{ex}(n,\pi)$, we have $S_n(U)$ is super-multiplicative and $L(U)$ exists and is finite, and $\textrm{ex}(n,U)$ is super-additive and $c(U)$ exists and is finite. By using Theorem \ref{mainlater} in the same way we deduced Corollary \ref{cor9}, we have the following corollary. 

\begin{corollary}
There is a family $U$ consisting of almost all $k$-permutations such that $c(U)$ exists and satisfies $c(U)=2^{\Omega\left((k/\log k)^{1/4}\right)}$. 
\end{corollary}

Cibulka's argument for $c(\pi)=O(L(\pi)^{4.5})$ also implies $c(U)=O(L(U)^{4.5})$, and hence Theorem \ref{mainfam} follows from the above corollary. Again, we could get a weaker bound with a simpler argument as above using Hall's matching theorem. 
 
Using Lemma \ref{lemexists} with $\ell=k^{1/2}$ and applying Theorem \ref{mainlater} with $q=\ell^{-5/6}$ and $r=\ell^{1/6}/4$, we obtain the following maybe surprising corollary showing that there is a very dense matrix of size exponential in a power of $k$ which avoids some $k$-permutation. 

\begin{corollary}\label{cor7}
For each $k$ there is a $k$-permutation $\pi$ and an $N \times N$ matrix $M$ with $N=2^{\Omega(k^{1/12})}$ such that the density of $M$ is at least $1-k^{-1/4}$ and $M$ avoids $\pi$. 
\end{corollary}

\section{Reducing counting to extremal problems}
\label{latersect}

Let $T_{n}(\pi)$ be the number of  $n \times n$ matrices which avoid $\pi$.  Klazar \cite{K00} showed that $T_{n}(\pi)=2^{\Theta(\textrm{ex}(n,\pi))}$. We next show his short argument. 
If $M$ is a  matrix which avoids $\pi$, then all matrics which are contained in $M$ also avoid $\pi$. Hence,  $T_n(\pi) \geq 2^{\textrm{ex}(n,\pi)}$. In the other direction, we have 
\begin{equation}\label{Kl}T_{2n}(\pi) \leq T_n(\pi)15^{\textrm{ex}(n,\pi)},\end{equation} 
which implies by induction on $n$ that $T_n(\pi) \leq 15^{\textrm{ex}(n,\pi)}$.  
The proof goes as follows. Consider a $2n \times 2n$ matrix $A$ which avoids $\pi$. Partition the set of rows and the set of columns into consecutive sets of size two, and consider the $n \times n$ matrix $B$ obtained by contracting these pairs. As $B$ is a contraction of $A$, and $A$ avoids $\pi$, then $B$ also avoids $\pi$. Thus, the number of possible choices for $B$ is at most $T_n(\pi)$. For each of the at most $\textrm{ex}(n,\pi)$ one-entries in $B$, there are 15 possible  $2 \times 2$ matrices which contract to get a one-entry. We therefore obtain (\ref{Kl}). 

The trivial estimate $S_n(\pi) \leq T_n(\pi)$ was used by Klazar in the proof that the F\"uredi-Hajnal conjecture implies the Stanley-Wilf conjecture. It only gives the estimate  $L(\pi) \leq 2^{O(c(\pi))}$. The following lemma will be used to give a simple proof of Cibulka's \cite{C09} improved estimate $L(\pi)=O(c(\pi)^2)$.

\begin{lemma}\label{general}
For a permutation $\pi$ and positive integers $n$ and $t$, letting $N=tn$, we have $$S_N(\pi) \leq T_n(\pi)t^{2N}.$$
\end{lemma}
\begin{proof}
Consider an $N \times N$ permutation matrix $A$ which avoids $\pi$. Partition the set of rows and the set of columns into consecutive sets of size $t$, and consider the $n \times n$ matrix $B$ obtained by contracting these intervals of size $t$. As $B$ is a contraction of $A$, and $A$ avoids $\pi$, then $B$ also avoids $\pi$. Thus, the number of possible choices for $B$ is at most $T_n(\pi)$. As $B$ came from the permutation matrix $A$ by contracting intervals of order $t$, each row of $B$ has at most $t$ one-entries.  After choosing $B$, the one-entry in a given row of $A$ must be in one of the blocks corresponding to a one-entry in $B$, giving at most $t^2$ choices for the location of the one in that row of $A$. Hence, the number of choices for $A$ which correspond to a given $B$ is at most $t^{2N}$. The desired upper bound on $S_N(\pi)$ follows.  
\end{proof}

Letting $t=c(\pi)$, and recalling $n=N/c(\pi)$, Lemma \ref{general} implies that $$S_N(\pi) \leq T_{n}(\pi)c(\pi)^{2N} \leq 2^{O(\textrm{ex}(n,\pi))}c(\pi)^{2N} \leq 2^{O\left(c(\pi)n\right)}c(\pi)^{2N}=\left(2^{O(1)}c(\pi)\right)^{2N}.$$ 
Taking the $N$th root of the above inequality,  we obtain $L(\pi)=O(c(\pi)^2)$.

\section{An improved upper bound} 
\label{sectnewthm}
For a matrix $P$, let $S_n(P)$ be the number of $n \times n$ permutation matrices which avoid $P$ as an interval minor. Let $m(n,P)$ be the maximum mass of an $n \times n$ matrix which avoids $P$ as an interval minor. Many of the results already discussed in this paper easily extend to give estimates on $S_n(P)$ and $m(n,P)$.  Note that, if $P$ is the permutation matrix of a permutation $\pi$, as containment of $P$ is equivalent to containment of $P$ as an interval minor, we have $S_n(P)=S_n(\pi)$ and $m(n,P)=m(n,\pi)$. 

We next provide a general framework extending that of Marcus and Tardos for proving upper bounds on $m(n,P)$. For a matrix $P$ and positive integers $s \leq t$, let $f_P(t,s)$ be the maximum $N$ such that there is an $N \times t$ matrix with at least $s$ ones in each row which avoids $P$ as an interval minor. If no such $N$ exists, we set $f_P(t,s)=\infty$. Similarly, let $g_P(t,s)$ be the minimum $N$ such any $t \times N$ matrix with at least $s$ ones in each column contains $P$ as an interval minor. If no such $N$ exists, we set $g_P(t,s)=\infty$. If $P$ is a symmetric matrix, then $f_P(t,s)=g_P(t,s)$. 

\begin{lemma} \label{verygen}
For positive integers $n,t,s$ with $s \leq t$ and a matrix $P$, we have the inequality $$m(tn,P) \leq m(s-1,P)m(n,P)+m(t,P)f_{P}(t,s)n+m(t,P)g_P(t,s)n.$$
\end{lemma}
\begin{proof}
Let $A$ be a $tn \times tn$ matrix which avoids $P$ as an interval minor. Partition the set of rows of $A$ into intervals of size $t$, and the set of columns of $A$ into intervals of size $t$, and contract these intervals to obtain an $n \times n$ matrix $B$. Since $B$ is a contraction of $A$, then $B$ also avoids $P$ as an interval minor. Call a $t \times t$ block, which is a product of one of the intervals of rows with one of the intervals of columns, {\it wide} if there are ones in at least $s$ of its columns, and {\it tall} if there are ones in at least $s$ of its rows. 

Each block of $A$ which is neither wide nor tall has ones in less than $s$ columns and in less than $s$ rows, and hence the submatrix of that block containing the rows and columns with at least one one-entry has at most $m(s-1,P)$ ones. As $B$ avoids $P$ as interval minor, $B$ has at most $m(n,P)$ ones, and hence the blocks of $A$ which are neither wide nor tall together have at most $m(s-1,P)m(n,P)$ ones. 

Each column of blocks of $A$ has at most $f_P(t,s)$ wide blocks. Indeed, contracting the rows of the wide blocks and deleting the rows of the blocks which are not wide, we obtain a $t \times N$ matrix which avoids $P$ as an interval minor, where $N$ is the number of wide blocks in that column, with at least $s$ ones in each row. Since this contraction also avoids $P$ as an interval minor, we have $N \leq f_P(t,s)$. Since there are $n$ columns of blocks, and each block has at most $m(t,P)$ ones in it, the total number of ones in wide blocks in $A$ is at most $m(t,P)f_P(t,s)n$. Similarly, the total number of ones in tall blocks in $A$ is at most $m(t,P)g_P(t,s)n$. Putting this all together, the mass of $A$ is at most $m(s-1,P)m(n,P)+m(t,P)f_{P}(t,s)n+m(t,P)g_P(t,s)n$, which completes the proof of the lemma. 
\end{proof}

Using the trivial inequalities $m(s-1,P) \leq (s-1)^2$ and $m(t,P) \leq t^2$,  the inequality in Lemma \ref{verygen} in the special case that $P=J_k$, $t=k^2$ and $s=k$ is $m(k^2n,J_k) \leq (k-1)^2 m(n,J_k)+2k^4f_{J_k}(k^2,k)n$. We have $f_{J_k}(k^2,k) \leq k{k^2 \choose k}$ from the pigeonhole principle, as any $k{k^2 \choose k}$ rows of length $k^2$ each with at least $k$ ones will contain $k$ rows with ones in exactly the same $k$ columns. This gives the inequality $$m(k^2n,J_k) \leq (k-1)^2 m(n,J_k)+2k^5{k^2 \choose k}n.$$ By induction on $n$, we obtain $m(n,J_k) \leq 2k^4{k^2 \choose k}n$. Noting that for a $k$-permutation $\pi$ we have $\textrm{ex}(n,\pi) \leq m(n,J_k)$, we obtain the Marcus-Tardos inequality (\ref{MTineq}) with the same proof. 

We next show how to improve this estimate. 

\begin{theorem}\label{mainab1}
We have $m(n,J_k) \leq 3k2^{8k}n$. 
\end{theorem}

If $\pi$ is a $k$-permutation, as $\textrm{ex}(n,\pi) \leq m(n,J_k)$, from Theorem \ref{mainab1}, we have $c(\pi) \leq 3k2^{8k}$. As Cibulka \cite{C09} obtained $L(\pi)=O(c(\pi)^2)$ (a short proof was given in the previous section), we obtain $L(\pi)=O(k^2 2^{16k})$. Hence, Theorem \ref{newthm} follows from Theorem \ref{mainab1}.
 
To obtain Theorem \ref{mainab1}, it will be helpful to consider $J_{r,k}$, the all ones $r \times k$ matrix. Let $f_{r,k}(t,s)=f_{J_{r,k}}(t,s)$. We have the following inequality. 

\begin{lemma}
If $s \leq t$ are positive integers with $t$ even, then $$f_{r,k}(t,s) \leq 2f_{r,k}(t/2,s)+2f_{r,k-1}(t/2,s/2).$$
\end{lemma}
\begin{proof}
Suppose we have an $N \times t$ matrix with at least $s$ ones in each row which avoids $J_{r,k}$ as an interval minor. Partition the $t$ columns into two intervals of $t/2$ columns. The number of rows where the first $t/2$ entries are all zero is at most $f_{r,k}(t/2,s)$, and the number of rows where the last $t/2$ entries are all zero is at most $f_{r,k}(t/2,s)$. The remaining rows have at least one one-entry in the first $t/2$ entries and at least one one-entry in the last $t/2$ entries. Of these rows, there are at most $f_{r,k-1}(t/2,s/2)$ rows that have at least $s/2$ one-entries in the last $t/2$ entries. Indeed, this can be seen by contracting the first $t/2$ columns, so the resulting submatrix has a one in the first entry of each row, which can be used to make one column of a $J_k$ interval minor. The remaining rows have at least $s/2$ one-entries in the first $t/2$ entries, and by the same argument, there are at most $f_{r,k-1}(t/2,s/2)$ such rows. Altogether, we get $N \leq 2f_{r,k}(t/2,s)+2f_{r,k-1}(t/2,s/2)$, which completes the proof. 
\end{proof}

We have the following lemma. 

\begin{lemma} \label{lablab}
For positive integers $s$, $t$ and $k$ with $t$ a power of $2$ and $2^{k-1} \leq s \leq t$, we have $$f_{r,k}(t,s) \leq r2^{k-1}t^2/s.$$
\end{lemma}
\begin{proof}
The proof is by induction on $k$ and $t$. In the base case $k=1$, we have $f_{r,1}(t,s)= r \leq r2^{k-1}t^2/s$, which follows from contracting the columns. Now suppose we know the lemma for all smaller choices of $k$ or for when $t' < t$. 
Then  $$f_{k,r}(t,s) \leq  2f_{r,k}(t/2,s)+2f_{r,k-1}(t/2,s/2) \leq 2r2^{k-1}(t/2)^2/s+2r2^{k-2}(t/2)^2/(s/2)=r2^{k-1}t^2/s.$$
This completes the proof by induction. \end{proof}

From Lemma \ref{verygen} with $P=J_k$, $s=2^{k-1}$ and $t=2^{2k}$ and using the trivial inequalities $m(s-1,P) \leq s^2$ and $m(t,P) \leq t^2$, noting that $P$ is symmetric, and using Lemma \ref{lablab} with $r=k$ we obtain 
$$ m(2^{2k}n,J_k) \leq s^2 m(n,J_k) +2t^2f_{k,k}(t,s)n \leq 2^{2k-2}m(n,J_k)+2k2^{8k}n.$$
Iterating this inequality, we obtain $$m(n,J_k) \leq 2k2^{6k}n(1+\frac{1}{4}+\frac{1}{4^2}\cdots)+m(2^{2k},J_k) \leq \frac{4}{3} 2k2^{8k}n+2^{4k} \leq 3k2^{8k},$$
which completes the proof of Theorem \ref{mainab1}.

\section{Concluding Remarks} 
\label{sectconc}

\vspace{0.1cm}
\noindent {\bf Permutations with large Stanley-Wilf limits}
\vspace{0.2cm}

 The following question of B\'ona seems quite interesting. 

\begin{question} \cite{B13+}\label{questBon}
What makes a $k$-permutation easier to avoid than another $k$-permutation? 
\end{question}

It was conjectured \cite{B05} that for each $k$ there is a layered $k$-permutation which is the easiest to avoid amongst the $k$-permutations, i.e., the Stanley-Wilf limit $L(\pi)$ is maximized amongst all $k$-permutations by a layered permutation. As discussed in the introduction, this conjecture is false. In fact, our results suggest in a certain sense that the opposite is true. Note that layered permutations are characterized by avoiding the small permutations $231$ and $312$.

A partial answer to Question \ref{questBon} appears to be that a $k$-permutation is easier to avoid if it contains all $t$-permutations with $t$ large. Indeed, Theorem \ref{mainlater} implies that if a $k$-permutation $\pi$ contains all $t$-permutations with $t=\omega((\log k)^{4})$, then $L(\pi)$ is super-polynomial in $k$. On the other hand, the following conjecture seems plausible. 

\begin{conjecture}
Fix $t$. If $\pi$ is a $k$-permutation which avoids some $t$-permutation, then $L(\pi)=k^{O(1)}$. 
\end{conjecture} 

\vspace{0.1cm}
\noindent {\bf Interval minors}
\vspace{0.2cm}

We have seen the usefulness of interval minors for studying extremal and counting problems for permuations. We think a further study of interval minors could be a fruitful direction for research. In particular, it would be interesting to obtain better estimates for $S_n(P)$ and $m(n,P)$. 

Another direction which could be quite rewarding:  What can be said about the structure of matrices which avoid a given matrix $P$ as an interval minor? In particular, it is interesting to investigate whether an analogue of the graph minor theory developed by Robertson and Seymour (see, e.g., \cite{RS}) could be established for interval minors. 

In this direction, Guillemot and Marx \cite{GM14} have introduced a new type of decomposition, and use this to give a linear-time algorithm for the permutation containment for a fixed permutation. Specifically, they show that determining whether an $n$-permutation contains a given $k$-permutation can be done in time $2^{O(k^2 \log k)}n$. Their proof relies on the Marcus-Tardos result \cite{MT}. As discussed by Guillemot and Marx, any improvement would give a faster algorithm for permutation containment. Our improved bound, Theorem \ref{mainab1}, can also easily be made into a linear time algorithm for finding a $J_k$ interval minor in a sufficiently dense matrix. It therefore implies the improved running time of $2^{O(k^2)}n$ for determining whether an $n$-permutation contains a given $k$-permutation. Our lower bound also provides a limitation to this method. 

\vspace{0.3cm}
\noindent {\bf Ramsey vs. extremal problems}
\vspace{0.2cm}

In this paper, we studied extremal and counting problems for permutation avoidance. Another natural question is to look at Ramsey problems for permutation avoidance. For a matrix $P$, define the minor Ramsey number $r(P)$ to be the minimum $n$ such that if the ones in $J_n$ are colored red and blue, then the red or the blue matrix contains $P$ as an interval minor. It is not difficult to show (see \cite{CFLS}) that for a $k \times k$-matrix $P$, $r(P) \leq k^2$. The Ramsey problem is quite different from the extremal problem, as we saw in  Corollary \ref{cor7} that we can make the red matrix almost complete (of density $1-k^{-1/4}$) and of exponential in a power of $k$ size such that it avoids some $k$-permutation matrix as an interval minor. 

\vspace{0.5cm}

\noindent {\bf Acknowledgement:} I would like to thank Amol Aggarwal, Noga Alon, Mikl\'os B\'ona, Sergi Elizalde, V\'it Jel\'inek, Joel Spencer, Richard Stanley, Einar Steingr\'imsson, Gabor Tardos, and Yufei Zhao for many helpful comments.


\begin{thebibliography}{}
\bibitem{AERWZ} 
M. H. Albert, M. Elder, A. Rechnitzer, P. Westcott, and M. Zabrocki, On the Stanley-Wilf limit of 4231-avoiding permutations and a conjecture of Arratia, {\it Adv. in Appl. Math.} {\bf 36} (2006), 96--105.

\bibitem{AF} 
N. Alon and E. Friedgut, On the number of permutations avoiding a given pattern, {\it J. Combin. Theory Ser. A} {\bf 89} (2000), 133--140.

\bibitem{Ang} 
O. Angel, A. E. Holroyd, G. Kozma, J. W\"astlund, and P. Winkler, Phase transition for dyadic tilings, http://arxiv.org/abs/1107.2636.

\bibitem{A99} 
R. Arratia, On the Stanley-Wilf conjecture for the number of permutations avoiding a given pattern, {\it Electron. J. Combin.} {\bf 6} (1999), N1, 4 pp. (electronic). 

\bibitem{B05}
M. B\'ona, The limit of a Stanley-Wilf sequence is not always rational, and layered patterns beat monotone patterns, {\it J. Combin. Theory Ser. A} {\bf 110} (2005), 223--235.

\bibitem{B07}
M. B\'ona, New records in Stanley-Wilf limits, {\it European J. Combin.} {\bf 28} (2007), 75--85. 

\bibitem{B12} 
M. B\'ona, {\bf Combinatorics of Permutations}, second edition, CRC Press - Chapman Hall, 2012.

\bibitem{B12+}
M. B\'ona, A new upper bound for 1324-avoiding permutations, preprint, arXiv:1207.2379.

\bibitem{B13+}
M. B\'ona, On the best upper bound for permutations avoiding a pattern of a given length, arXiv:1209.2404.

\bibitem{C09} 
J. Cibulka, On constants in the F\"uredi-Hajnal and the Stanley-Wilf conjecture, {\it J. Combin. Theory Ser. A}  {\bf 116} (2009), 290--302.

\bibitem{CJS}
A. Claesson, V. Jel\'inek, and E. Steingr\'imsson, Upper bounds for the Stanley-Wilf limit of 1324 and other layered patterns, {\it J. Combin. Theory Ser. A} {\bf 119} (2012), 1680--1691. 

\bibitem{CLSW} E. G. Coffman, G. S. Lueker, J. Spencer, and P. Winkler, Packing random rectangles, {\it Probab. Theory Related Fields} {\bf 120} (2001), 585--599.

\bibitem{CFLS} 
D. Conlon, J. Fox, C. Lee, B. Sudakov, Ordered Ramsey numbers, in preparation. 

\bibitem{FH} Z. F\"uredi and P. Hajnal, Davenport-Schinzel theory of matrices, {\it Discrete Math.} {\bf 103} (1992), 231--251.


\bibitem{GM14} S. Guillemot and D. Marx, Finding Small Patterns in Permutations in Linear Time, {\it Proc. SODA 2014}, to appear.

\bibitem{JRS} S. Janson, D. Randall, and J. Spencer, Random dyadic tilings of the unit square, {\it Random Structures Algorithms} {\bf 21} (2002), 225--251. 

\bibitem{KK02} 
T. Kaiser, M. Klazar, On growth rates of closed permutation classes, {\it Electron. J. Combin.} {\bf 9} (2002/03), Research paper 10, 20 pp. (electronic)

\bibitem{Ki1} 
S. Kitaev, {\bf Patterns in Permutations and Words}, Springer, 2011. 

\bibitem{K00} M. Klazar, The F\"uredi-Hajnal conjecture implies the Stanley-Wilf conjecture, in: D. Krob, A.A. Mikhalev, A.V. Mikhalev (Eds.), Formal Power Series and Algebraic Combinatorics, Springer, Berlin, 2000, pp. 250--255.

\bibitem{K68} D. Knuth, {\bf The art of computer programming, vol. 1: fundamental algorithms}, Addison-Wesley, 1968. 

\bibitem{MT} A. Marcus and G. Tardos, Excluded permutation matrices and the Stanley-Wilf conjecture. {\it J. Combin. Theory Ser. A} {\bf 107} (2004), 153--160.

\bibitem{M15} 
P. A. MacMahon, Combinatory Analysis, Cambridge University Press, 1915. 

\bibitem{N11} J. Novak, An asymptotic version of a theorem of Knuth, {\it Adv. in Appl. Math.} {\bf 47} (2011),  49--56. 

\bibitem{PT06} 
J. Pach and G. Tardos, Forbidden paths and cycles in ordered graphs and matrices, {\it Israel J. Math.} {\bf 155} (2006), 359--380.

\bibitem{R81} 
A. Regev, Asymptotic values for degrees associated with strips of Young diagrams, {\it Adv. Math.} {\bf 41} (1981), 115--136. 

\bibitem{RS}
N. Robertson and P. D. Seymour, Graph Minors. XX. Wagner's conjecture, {\it J. Combin. Theory Ser. B} {\bf 92} (2004), 325--357.

%\bibitem{SS85} R. Simion and F. W. Schmidt, Restricted permutations, {\it Europ. J. Combinatorics}, {\bf 6} (1985), 383--406.

\bibitem{S06} 
R. P. Stanley, Increasing and decreasing subsequences and their variants, in: Proceedings of the International Congress of Mathematicians, Plenary Lectures, vol. I, Madrid, Spain, 2006, pp. 545--579.

\bibitem{St08} 
R. P. Stanley, Longest alternating subsequences of permutations, {\it Michigan Math. J.} {\bf 57} (2008), 675--687.

\bibitem{Notices} 
R. P. Stanley, Herb Wilf and Pattern Avoidance, {\it Not. Amer. Math. Soc.}, to appear. 

\bibitem{Ste13}
E. Steingr\'imsson, Some open problems on permutation patterns, in: Surveys in combinatorics 2013, London
Math. Soc. Lecture Note Ser. Cambridge Univ. Press, Cambridge, 2013, 239--263.

\bibitem{Te} 
B. Tenner, {\it Database of Permutation Pattern Avoidance}, \newline  \url{math.depaul.edu/bridget/cgi-bin/dppa.cgi}

\bibitem{W90} 
J. West, Permutations with forbidden subsequences and stack-sortable permutations, Ph.D. Thesis,  M.I.T. 1990.


\end{thebibliography}
\end{document}